\documentclass{article} 
\usepackage{graphicx}
\usepackage{authblk}

\usepackage{macros,stmaryrd,hyperref,tabularx,amsmath, amssymb}
\usepackage{float}
\usepackage[inline]{enumitem}
\usepackage{amsthm}
\newtheorem{theorem}{Theorem}
\newtheorem{lemma}{Lemma}
\newtheorem{definition}{Definition}
\newtheorem{proposition}{Proposition}
\newtheorem{corollary}{Corollary}
\newtheorem{question}{Question}

\usepackage{tikz}
\usetikzlibrary{arrows,backgrounds,quotes,calc}

\newcommand{\fullproof}[1]{}

\makeatletter
\newcommand*{\ineq}[2][]{%
  \begingroup
    \refstepcounter{equation}%
    \ifx\\#1\\%
    \else
      \label{#1}%
    \fi
    \relpenalty=10000 %
    \binoppenalty=10000 %
    \@eqnnum \ \ensuremath{%
      #2%
    }%
  \endgroup
}
\makeatother

\newcommand{\seq}{\succcurlyeq}
\newcommand{\Var}{{\mathbb P}}
\newcommand{\U}{U}

\newcommand{\length}[1]{\lvert#1\rvert}

\newcommand{\ignore}[1]{}
\newcommand{\peq}{\preccurlyeq}

\newcommand{\sfun}{\msuccfct}
\newcommand{\itlht}{{\sf ITL^{ht}}}
\newcommand{\iltl}{{\sf ITL^e}}
\newcommand{\itlb}{{\sf ITL^p}}
\newcommand{\lang}{{\mathcal L}}

\newcommand{\Qbase}[1][]{}  % for the initial (base) model
\newcommand{\Qstrat}[1][]{^{{\str}\ifx#1\relax \relax\else, #1\fi}}
\newcommand{\str}{{\rm e }} % for the stratified model

\usepackage{xypic}
\usepackage{graphicx}
\usepackage{multicol}
\def\lang{{\cal L}}

\usepackage{MnSymbol}
\usepackage{subcaption}
\renewcommand{\tnext}{\medcircle}
\def\R{\mathop{\cal R}}
\def\U{\mathop{\cal U}}
\def\bisim{\mathrel{\cal Z}}

\title{Bisimulations for intuitionistic temporal logics\thanks{This research was partially supported by ANR-11-LABX-0040-CIMI within the program ANR-11-IDEX-0002-02.}}
\author[1]{Philippe Balbiani\footnote{\href{mailto:philippe.balbiani@irit.fr}{\tt philippe.balbiani@irit.fr}}}
\author[1]{Joseph Boudou\footnote{\href{mailto:joseph.boudou@irit.fr}{\tt joseph.boudou@irit.fr}}}
\author[1,2]{Mart\'in Di\'eguez\footnote{\href{mailto:martin.dieguez@enib.fr}{\tt martin.dieguez@enib.fr}}}
\author[1,3]{David Fern\'andez-Duque\footnote{\href{mailto:david.fernandezduque@ugent.be}{\tt david.fernandezduque@ugent.be}}}
\affil[1]{IRIT, Toulouse University. Toulouse, France}
\affil[2]{CERV, ENIB,LAB-STICC. Brest, France}
\affil[3]{Department of Mathematics, Ghent University. Ghent, Belgium}

\begin{document}
\maketitle

\begin{abstract} 
We introduce bisimulations for the logic $\iltl$ with $\tnext$ (`next'), $\U$ (`until') and $\R$ (`release'), an intuitionistic temporal logic based on structures $(W,\peq,\msuccfct)$, where $\peq$ is used to interpret intuitionistic implication and $\msuccfct$ is a $\peq$-monotone function used to interpret the temporal modalities. Our main results are that $\diam$ (`eventually'), which is definable in terms of $\U$, cannot be defined in terms of $\tnext$ and $\ubox$, and similarly that $\ubox$ (`henceforth'), definable in terms of $\R$, cannot be defined in terms of $\tnext$ and $\U$, even over the smaller class of here-and-there models.
\end{abstract}
%\begin{keyword}
%linear temporal logic, bisimulation, expressivity
%\end{keyword}

\section{Introduction}

The definition and study of full combinations of modal~\cite{Chagrov1997} and intuitionistic~\cite{DVD86,MInt} logics can be quite challenging~\cite{Simpson94}, and temporal logics, such as $\sf LTL$ \cite{pnueli}, are no exception. Some intuitionistic analogues of temporal logics have been proposed, including logics with `past' and `future' tenses \cite{Ewald} or with `next' \cite{Davies96,KojimaNext}, and `henceforth' \cite{KamideBounded}. We proposed an alternative formulation in \cite{Boudou2017}, where we defined the logics $\iltl$ and $\itlb$ using semantics similar to those of {\em expanding} and {\em persistent} products of modal logics, respectively \cite{pml}, and the tenses $\tnext$ (`next'), $\diam$ (`eventually'), and $\ubox$ (`henceforth'). $\iltl$ in particular differs from previous proposals (e.g.~\cite{Ewald,PS86}) in that we consider minimal frame conditions that allow for all formulas to be upward-closed under the intuitionistic preorder, which we denote $\peq$. We then showed that $\iltl$ with $\tnext$ (`next'), $\diam$ (`eventually'), and $\ubox$ (`henceforth') is decidable, thus obtaining the first intuitionistic analogue of $\sf LTL$ which contains the three tenses, is conservative over propositional intuitionistic logic, is interpreted over unbounded time, and is known to be decidable.

Note that both $\diam$ and $\ubox$ are taken as primitives, in contrast with the classical case, where $\diam \varphi$ may be defined by $\diam\varphi\equiv \neg \ubox \neg \varphi$, whereas the latter equivalence is not intuitionistically valid. The same situation holds in the more expressive language with $\U$ (`until'): while the language with $\tnext$ and $\U$ is equally expressive to classical monadic first-order logic with $\leq$ over $\mathbb N$~\cite{Gabbay80}, $\U$ admits a first-order definable intuitionistic dual, $\R$ (`release'), which cannot be defined in terms of $\U$ using the classical definition. However, this is not enough to conclude that $\R$ cannot be defined in a different way. Thus, while in \cite{Boudou2017} we explored the question of decidability, here we will focus on {\em definability;} which of the modal operators can be defined in terms of the others? As is well-known, $\diam \varphi \equiv \top \U \varphi$ and $\ubox \varphi \equiv \bot \R \varphi$; these equivalences remain valid in the intuitionistic setting. Nevertheless, we will show that $\ubox$ cannot be defined in terms of $\U$, and $\diam$ cannot be defined in terms of $\R$; in order to prove this, we will develop a theory of bisimulations on $\iltl$ models.

Following Simpson \cite{Simpson94} and other authors, we interpret the language of $\iltl$ using bi-relational structures, with a partial order $\peq$ to interpret intuitionistic implication, and a function or relation, which we denote $\msuccfct$, representing the passage of time. Alternatively, one may consider topological interpretations \cite{Davoren2009}, but we will not discuss those here. Various intuitionistic temporal logics have been considered, using variants of these semantics and different formal languages. The main contributions include:

\begin{itemize}\itemsep0pt	
	\item Davies' intuitionistic temporal logic with $\tnext$ \cite{Davies96} was provided Kripke semantics and a complete deductive system by Kojima and Igarashi \cite{KojimaNext}.	
	\item Logics with $\tnext,\ubox$ were axiomatized by Kamide and Wansing \cite{KamideBounded}, where $\ubox$ was interpreted over bounded time.	
	\item Nishimura \cite{NishimuraConstructivePDL} provided a sound and complete axiomatization for an intuitionistic variant of the propositional dynamic logic $\sf PDL$.	
	\item Balbiani and Di\'eguez axiomatized the here-and-there variant of $\sf LTL$ with $\tnext,\diam,\ubox$ \cite{BalbianiDieguezJelia}, here denoted $\itlht$.	
	\item Fern\'andez-Duque \cite{DFD2016} proved the decidability of a logic based on topological semantics with $\tnext,\diam$ and a universal modality.	
	
	\item The authors \cite{Boudou2017} proved that the logic $\iltl$ with $\tnext,\diam,\ubox$ has the strong finite model property and hence is decidable, yet the logic $\itlb$, based on a more restrictive class of frames, does not enjoy the fmp.
\end{itemize}

In this paper, we extend $\iltl$ to include $\U$ (`until') and $\R$ (`release'). We will introduce different notions of bisimulation which preserve formulas with $\tnext$ and each of  $\diam$, $\ubox$, $\U$ and $\R$. With this, we will show that $\R$ (or even $\ubox$) may not be defined in terms of $\U$ over the class of here-and-there models, while $\diam$ {\em can} be defined in terms of $\ubox$ and $\U$ can be defined in terms of $\R$ over this class. However, we show that over the wider class of expanding models, $\diam$ cannot be defined in terms of $\ubox$.

\section{Syntax and semantics} \label{SecSynSem}

We will work in sublanguages of the language $\lang$ given by the following grammar:
\begin{equation*}
  \varphi, \psi \gramdef p \alt \bot \alt
    \varphi \wedge \psi \alt \varphi \vee \psi \alt \varphi \imp \psi \alt
    \tnext \varphi \alt \diam\varphi \alt \ubox\varphi\alt \varphi \U \psi \alt \varphi \R \psi
\end{equation*}
where $p$ is an element of a countable set of propositional variables $\Var$. All sublanguages we will consider include all Boolean operators and $\tnext$, hence we denote them by displaying the additional connectives as a subscript; for example, $\lang_{\diam\ubox}$ denotes the $\U$-free, $\R$-free fragment. As an exception to this general convention, $\lang_{\tnext}$ denotes the fragment without $\diam,\ubox,{\U}$ or $\R$.

Given any formula $\varphi $,
we define the \emph{length} of $\varphi$ (in symbols, $\length{\varphi}$) recursively as follows:

\begin{itemize}[itemsep=0pt]
	\item $\length{p} = \length{\bot} = 0$;
	\item $\length{\phi \odot \psi}= 1 + \length{\phi} + \length{\psi}$, with $\odot \in \lbrace \vee, \wedge,\imp, \R, \U \rbrace$;	
	\item $\length{\odot \psi}= 1 + \length{\psi}$, with $\odot \in \lbrace \neg, \tnext, \ubox, \diam \rbrace$.
\end{itemize}
\noindent Broadly speaking, the length of a formula $\varphi$ corresponds to the number of connectives appearing in $\varphi$.

\subsection{Dynamic posets}

Formulas of $\lang$ are interpreted over dynamic posets. A {\em dynamic poset} is a tuple $\mathcal D = (W,\peq,\sfun)$,
where 
$W$ is a non-empty set of states,
$\irel$ is a partial order,
and
$\sfun$ is a function from $W$ to $W$ satisfying the {\em forward confluence} condition that for all $ w, v \in W, $ if $ w \irel v $ then $ \msuccfct(w) \irel \msuccfct(v).$ An {\em intuitionistic dynamic model,} or simply {\em model,} is a tuple $\modelbase$ consisting of a dynamic poset equipped
with a valuation function $V$ from $W$ to sets of propositional variables that is {$\peq$-monotone,} in the sense that
for all $ w, v \in W, $ if $ w \irel v $ then $ V(w) \subseteq V(v).$ In the standard way, we define $\msuccfct[^0](w) = w$ and,
for all $k > 0$, $\msuccfct[^k](w) = \msuccfct\left(\msuccfct[^{k-1}](w)\right)$. Then we define the satisfaction relation $\sat$ inductively by:

\begin{multicols}2
\noindent\begin{enumerate}[itemsep=0pt]
	\item $\model, w \sat p $  iff $p \in V(w) $;
	\item $\model, w \nsat \bot$;
	\item $\model, w \sat \varphi\wedge \psi$  iff $  \model, w \sat \varphi $ and $\model, w \sat \psi$;
	\item $\model, w \sat \varphi \vee \psi $  iff  $ \model, w \sat \varphi $ or $\model, w \sat \psi$;
	\item $\model, w \sat \tnext \varphi $   iff $ \model, \msuccfct(w) \sat \varphi $;
		\item  $\model, w \sat \varphi \imp \psi $ iff $\forall v \succcurlyeq w $, if $\model, v \sat \varphi$, then $\model, v \sat \psi $;
		\columnbreak
	\item $\model, w \sat \diam \varphi $  iff there exists $ k \text{ s.t. } \model, \msuccfct[^k](w) \sat \varphi$;
	\item $\model, w \sat \ubox \varphi $  iff for all $ k $, $ \model, \msuccfct[^k](w) \sat \varphi$;
	\item  $\model, w \sat \varphi \U \psi $ iff there exists $k \ge 0 \hbox{ s.t. } \model, \msuccfct[^k](w) \sat \psi$ and $\forall i \in [0,k)$, $\model, \msuccfct[^i](w) \sat \varphi$;	
	\item 	$\model, w \sat \varphi \R \psi $ iff for all $k \ge 0$, either $\model, \msuccfct[^k](w) \sat \psi$, or $ \exists i \in [ 0 , k)$ s.t.~$\model, \msuccfct[^i](w) \sat \varphi$.
\end{enumerate}
\end{multicols}
As usual, a formula $\varphi$ is {\em satisfiable over a class of models $\Omega$} if there is a model $\model\in \Omega$ and a world $w$ of $\model$ so that $\model,w\sat\varphi$, and {\em valid over $\Omega$} if, for every world $w$ of every model $\model\in \Omega$, $\model,w\sat\varphi$. Satisfiability (validity) over the class of models based on an arbitrary dynamic poset will be called {\em satisfiability (validity)} for $\iltl$, or {\em expanding domain linear temporal logic.}\footnote{Note that in \cite{Boudou2017} we used `$\iltl$' to denote the fragment of this logic without $\U,\R$.}

The relation between dynamic posets and expanding products of modal logics is detailed in \cite{Boudou2017}, where the following is also shown. Below, we use the notation $\llbracket\varphi\rrbracket=\{w\in W \mid \model,w\sat \varphi \}.$

\begin{lemma}
Let $\mathcal D=(W,{\irel},\msuccfct)$, where $(W,{\irel})$ is a poset and $\msuccfct\colon W\to W$ is any function. Then, $\mathcal D$ is a dynamic poset if and only if, for every valuation $V$ on $W$ and every formula $\varphi$, $\llbracket \varphi \rrbracket$ is {\em $\irel$-monotone}, i.e., if $w\in \llbracket \varphi\rrbracket$ and $v\seq w$, then $v \in \llbracket \varphi \rrbracket$.
\end{lemma}

The proof that all valuations on a dynamic poset are $\peq$-monotone proceeds by a standard structural induction on formulas, and the cases for $\U,\R$ are similar to those for $\diam,\ubox$ in \cite{Boudou2017}. This suggests that dynamic posets provide suitable semantics for intuitionistic $\sf LTL$. Moreover, dynamic posets are convenient from a technical point of view:

\begin{theorem}[\cite{Boudou2017}]\label{TheoBound}
There exists a computable function $B$ such that any formula $\varphi \in \lang_{\diam\ubox}$ satisfiable (resp.~falsifiable) on an arbitrary model is satisfiable (resp.~falsifiable) on a model whose size is bounded by $B(\nos\varphi)$.
\end{theorem}

It follows that the $\lang_{\diam\ubox}$-fragment of $\iltl$ is decidable. Moreover, as we will see below, many of the familiar axioms of classical $\sf LTL$ are valid over the class of dynamic posets, making them a natural choice of semantics for intuitionistic $\sf LTL$. 

\subsection{Persistent posets}

Despite the appeal of dynamic posets, in the literature one typically considers a more restrictive class of frames, similar to persistent frames, as we define them below.

\begin{definition}
Let $(W,{\preccurlyeq})$ be a poset.
If $\msuccfct\colon W\to W$ is such that,
whenever $v \succcurlyeq \sfun(w)$, there is $u\succcurlyeq w$ such that $v=\sfun (u)$,
we say that $\sfun$ is {\em backward confluent}.
If $\sfun$ is both forward and backward confluent, we say that it is {\em persistent}.
A tuple $(W,{\preccurlyeq},\msuccfct)$ where $\msuccfct$ is persistent is a {\em persistent
intuitionistic temporal frame}, and the set of valid formulas over the class of persistent
intuitionistic temporal frames is denoted $\itlb$, or {\em persistent domain $\sf LTL$}.
\end{definition}

As we will see, persistent frames do have some technical advantages over arbitrary dynamic posets. Nevertheless, they have a crucial disadvantage:

\begin{theorem}[\cite{Boudou2017}]\label{TheoPersNonFin}
The logic $\itlb$ does not have the finite model property, even for formulas in $\mathcal L_{\diam\ubox}$.
\end{theorem}

\subsection{Temporal here-and-there models}\label{SecHT}

An even smaller class of models which, nevertheless, has many applications is that of temporal here-and-there models \cite{BalbianiDieguezJelia}. Some of the results we will present here apply to this class, so it will be instructive to review it. Recall that the logic of here-and-there is the maximal logic strictly between classical and intuitionistic propositional logic, given by a frame $\{0,1\}$ with $0 \peq 1$. The logic of here-and-there is obtained by adding to intuitionistic propositional logic the axiom $p \vee ( p \imp q )\vee \neg q.$

A temporal here-and-there frame is a persistent frame that is `locally' based on this frame. To be precise:

\begin{definition}
A {\em temporal here-and-there frame} is a persistent frame $(W,{\peq},S)$ such that $W = T \times \{0,1\}$ for some set $T$, and there is a function $f\colon T\to T$ such that for all $t,s \in T$ and $i,j\in \{ 0,1\}$, $(t,i) \peq (s,j)$ if and only if $t=s$ and $i\leq j$ and $S(t,i) = (f(t),i)$.
\end{definition}

The prototypical example is the frame $(W,{\peq},S)$, where $W = \mathbb N\times \{0,1\}$, $(i,j) \peq (i',j')$ if $i = i'$ and $j \leq j'$, and $S(i,j) = (i+1,j)$. Note, however, that our definition allows for other examples (see Figure \ref{FigBoxU}). In \cite{BalbianiDieguezJelia}, this logic is axiomatized, and it is shown that $\ubox$ cannot be defined in terms of $\diam$, a result we will strengthen here to show that $\ubox$ cannot be defined even in terms of $\U$.
It is also claimed in \cite{BalbianiDieguezJelia} that $\diam$ is not definable in terms of $\ubox$ over the class of here-and-there models, but as we will see in Proposition \ref{propDiamDefinHT}, this claim is incorrect.

\section{Some valid and non-valid $\iltl$-formulas}

In this section we explore which axioms of classical $\sf LTL$ are still valid in our setting. We start by showing that the intuitionistic version of the interaction and induction axioms used in~\cite{BalbianiDieguezJelia} remain valid in our setting. However, not all Fisher-Servi axioms~\cite{FS84}, which are valid in the here-and-there $\sf LTL$ of~\cite{BalbianiDieguezJelia}, are valid in $\iltl$.

\begin{proposition}\label{prop:valid:iltl}
The following formulas:
\begin{multicols}{2}
\begin{enumerate}[itemsep=0pt]
\item $\tnext\bot\leftrightarrow \bot$
\item $\tnext \left( \varphi \wedge \psi \right) \leftrightarrow \left(\tnext \varphi \wedge\tnext \psi\right)$;
\item $\tnext \left( \varphi \vee \psi \right) \leftrightarrow \left(\tnext \varphi \vee\tnext \psi\right)$;
\item $\tnext\left( \varphi \rightarrow \psi \right) \rightarrow \left(\tnext\varphi \rightarrow \tnext\psi\right)$;
\item $\ubox \left( \varphi \rightarrow \psi \right) \rightarrow \left(\ubox \varphi \rightarrow \ubox \psi\right)$;
\item $\ubox \left( \varphi \rightarrow \psi \right) \rightarrow \left(\diam \varphi \rightarrow \diam \psi\right)$;
\item $\diam \left( \varphi \vee \psi \right) \rightarrow \left(\diam \varphi \vee\diam \psi\right)$;
\item $\ubox \varphi \leftrightarrow \varphi \wedge \tnext \ubox \varphi$;
\item $\varphi \vee \tnext \diam \varphi \leftrightarrow \diam \varphi$;
\item\label{ax:ind:1} $\ubox \left( \varphi \rightarrow \tnext \varphi \right) \rightarrow \left( \varphi \rightarrow \ubox \varphi \right)$
\item\label{ax:ind:2} $\left(\diam \varphi \imp \varphi \right) \rightarrow  \left(\tnext \varphi \rightarrow \varphi\right)$.
\end{enumerate}
\end{multicols}
\noindent are $\iltl$-valid.
\end{proposition}
\begin{proof}
Let us consider~\eqref{ax:ind:1} and~\eqref{ax:ind:2}. For \eqref{ax:ind:1}, let $\model=(W,{\irel},{\msuccfct})$ be any $\iltl$ model and $w\in W$ be such that $\model, w \sat \ubox \left( \varphi \imp \tnext \varphi\right)$. Let $v\succcurlyeq w$ be arbitrary and assume that $\model, v \sat \varphi$. Then, by induction on $i$ we obtain that $\msuccfct[^i](w)\irel \msuccfct[^i](v)$ for all $i$; since $\model, \msuccfct[^i](w) \sat ¿ \varphi \imp \tnext \varphi$ for all $i$, it follows that $\model, \msuccfct[^i](v) \sat ¿ \varphi \imp \tnext \varphi$ for all $i$ as well. Hence an easy induction shows that $\model, \msuccfct[^i](v) \sat ¿ \varphi $ for all $i$, which means that $\model,v\sat \ubox\varphi$. Since $w$ was arbitrary, we conclude that the formula \eqref{ax:ind:1} is valid.

For \eqref{ax:ind:2}, let $\model$ be as above and $w\in W$ be such that $\model, w \sat \left( \diam \varphi \imp \varphi\right)$. Let $v\succcurlyeq w$ be such that $\model, v \sat \tnext \varphi$. It follows that $\model, v \sat \diam \varphi$, so $\model, v \sat \varphi$. Since $w,v$ were arbitrary, the formula \eqref{ax:ind:2} is valid as well.

The proofs for the rest of formulas are left to the reader.
\end{proof}

Some of the well-known Fisher Servi axioms~\cite{FS84} are only valid on the class of persistent frames.

\begin{proposition}\label{prop:nvalid:iltl}
The formulas 
\begin{multicols}{2}
\begin{enumerate}[itemsep=0pt]
	\item\label{ItFSOne} $\left(\tnext \varphi \rightarrow \tnext \psi \right) \imp \tnext \left(\varphi \rightarrow \psi\right)$,
	\item\label{ItFSTwo} $\left( \diam \varphi \rightarrow \ubox \psi \right) \imp \ubox \left(\varphi \rightarrow \psi\right)$
\end{enumerate}
\end{multicols}
\noindent are not $\iltl$-valid. However they are $\itlb$-valid.
\end{proposition}
\begin{proof}
%\begin{displaymath}
%\begin{gathered}\xymatrix@C=0.8cm{ && \stackrel{u}{\lbrace p \rbrace} \ar@(ur,dr)^{\msuccfct} &&\\
%	\stackrel{w}{\emptyset}\ar[rr]^{\msuccfct}&&  \stackrel{v}{\emptyset }\ar@{.>}[u]^{\irel}\ar@(ur,dr)^{\msuccfct}&& }\end{gathered}
%\end{displaymath}
Let $\lbrace p,q \rbrace$ be a set of propositional variables and let us consider the $\iltl$ model $\model=\left(W,\irel,\msuccfct,V\right)$ defined as:
\begin{enumerate*}[itemsep=0pt,label=\arabic*)]
	\item $W = \lbrace w,v,u\rbrace$;
	\item $\msuccfct(w) = v$, $\msuccfct(v) = v$ and $\msuccfct(u) = u$;
	\item $v\irel u$;
	\item $V(p) = \lbrace u \rbrace$.	
\end{enumerate*}	
\noindent Clearly, $\model, u \not \sat p \imp q$, so $\model, v \not \sat p \imp q$. By definition, $\model, w \not \sat \tnext \left(p \imp q\right)$ and $\model, w \not \sat \ubox \left(p \imp q\right)$; however, it can easily be checked that $\model, w  \sat \tnext p \imp \tnext q$ and $\model, w  \sat \diam p \imp \ubox q$, so $\model, w  \not \sat \left(\tnext p \rightarrow \tnext q \right) \rightarrow \tnext \left( p \rightarrow q \right)$ and $\model, w  \not \sat \left( \diam p \rightarrow \ubox q \right) \rightarrow \ubox \left(p \rightarrow q\right)$.

Let us check their validity over the class of persistent frames. For \eqref{ItFSOne}, let $\model = (W,{\peq},S,V)$ be an $\itlb$ model and $w$ a world of $\model$ such that $\model, w \sat \tnext \varphi \imp \tnext \psi$. Suppose that $v\seq \msuccfct(w)$ satisfies $\model, v \sat \varphi$. By backward confluence, there exists $u\seq w$ such that $v = \msuccfct(u)$, so that $\model,u \sat \tnext \varphi$ and thus $\model,u \sat \tnext \psi$. But this means that $\model, v \sat \psi$, and since $v \seq S(w)$ was arbitrary, $\model, S(w) \sat  \varphi \imp  \psi$, i.e.~$\model, w \sat \tnext (\varphi \imp  \psi ) $.

Similarly, for \eqref{ItFSTwo} let us assume that $\model  = (W,{\peq},S,V) $ is an $\itlb$ model and $w$ a world of $\model$ such that $\model, w \sat \diam \varphi \imp \ubox \psi$. Consider arbitrary $k\in \mathbb N$, and suppose that $v\seq S^k(w)$ is such that $\model,v \sat \varphi$. Then, it is readily checked that the composition of backward confluent functions is backward confluent, so that in particular $S^k$ is backward confluent. This means that there is $u\seq w$ such that $S^k(u) = v$. But then, $\model,u\sat \diam \varphi$, hence $\model,u\sat \ubox \psi$, and $\model,v\sat \psi$. It follows that $\model,S^k(w) \sat \varphi\to \psi$, and since $k$ was arbitrary, $\model,w\sat \ubox(\varphi\imp \psi)$.
\end{proof}

We make a special mention of the schema $\ubox \left(\ubox \varphi \rightarrow \psi \right) \vee \ubox \left(\ubox \psi \rightarrow \varphi \right)$, which characterises the class of \emph{weakly connected frames}~\cite{G92} in classical modal logic. We say that a frame $\left(W,R,V\right)$ is weakly  connected iff it satisfies the following first-order property: for every $x,y,z \in W$, if $ x \mathrel R y$ and $x\mathrel R z$, then either $y\mathrel R z$, $y = z$, or $z\mathrel R y$.

\begin{proposition}
The axiom schema $\ubox \left(\ubox \varphi \rightarrow \psi \right) \vee \ubox \left(\ubox \psi \rightarrow \varphi \right)$ is not $\itlht$-valid.
\end{proposition}
\begin{proof}	
	Let us consider the set of propositional variables $\lbrace p,q\rbrace$ and the $\itlht$ model $\model= \left(W,\irel,S,V\right)$ defined as:
	\begin{enumerate*}[label=\arabic*)]
		\item $W = \lbrace w,t,u,v\rbrace$;
		\item $\msuccfct(w) = v$, $\msuccfct(v) = v$, $\msuccfct(t)=u$ and $\msuccfct(u)=u$;
		\item $v \irel u$ and $w \irel t$;
		\item $V(p)=\lbrace v,u\rbrace$ and $V(q) = \lbrace t,u \rbrace$.
	\end{enumerate*}	
%	\begin{displaymath}
%	\begin{gathered}\xymatrix@R=0.3cm{\stackrel{t}{\lbrace q \rbrace}\ar[rr]^{\msuccfct} && \stackrel{u}{\lbrace p,q \rbrace} \ar@(ur,dr)^{\msuccfct} &&\\
%		\stackrel{w}{\emptyset}\ar@{.>}[u]^{\irel}\ar[rr]^{\msuccfct}&&  \stackrel{v}{\lbrace p \rbrace}\ar@{.>}[u]^{\irel}\ar@(ur,dr)^{\msuccfct}&& }\end{gathered}
%	\end{displaymath}
%	
\noindent The reader can check that $\model, v \not \sat \ubox p \imp q $ and $\model, t \not \sat  \ubox q \imp p $. Consequently $\model, w \not \sat \ubox \left( \ubox p \imp q \right)\vee \ubox \left( \ubox q \imp p \right)$.
\end{proof}

 Finally, we show that $\diam \varphi$ (resp. $\ubox \varphi$) can be defined in terms of $\U$ (resp. $\R$) and the $\sf LTL$ axioms involving $\U$ and $\R$ are also valid in our setting:

\begin{proposition}\label{PropUValid}
The following formulas are $\iltl$-valid:
\begin{multicols}{2}
\begin{enumerate}[itemsep=0pt]
	\item\label{ItUOne} \mbox{$\varphi \U \psi \leftrightarrow \psi \vee \left( \varphi \wedge  \tnext \left(\varphi \U \psi \right)\right)$;}
	\item\label{ItUTwo} \mbox{$\varphi \R \psi \leftrightarrow \psi \wedge \left( \varphi \vee  \tnext \left(\varphi \R \psi \right)\right)$;}
	\item\label{ItUThree} \mbox{$\varphi \U \psi \imp \diam \psi$;}
	\item\label{ItUFour} \mbox{$\ubox \psi \imp \varphi\R\psi$;}
	\item\label{ItUFive} \mbox{$\diam \varphi \leftrightarrow \top \U \varphi$;}
	\item\label{ItUSix} \mbox{$\ubox \varphi \leftrightarrow \bot \R \varphi$;}
\item 	\label{ItUSeven} \mbox{$\tnext(\varphi \U \psi)\leftrightarrow {\tnext\varphi}\U{\tnext\psi}$;}
 \item 	\label{ItUEight} \mbox{$\tnext(\varphi \R \psi)\leftrightarrow {\tnext\varphi}\R{\tnext\psi}$.}
	
\end{enumerate}
\end{multicols}
\end{proposition}
\begin{proof}
We consider some cases below. For \eqref{ItUOne}, from left to right, let us assume that $\model, w \sat \varphi \U \psi$. Therefore there exists $k \ge 0$ s.t. $\model, \msuccfct[^k](w) \sat \psi$ and for all $j$ satisfying $0 \le j < k$, $\model, \msuccfct[^j](w) \sat \varphi$. If $k = 0$  then $\model, w\sat \psi$ while, if $k > 0$ it follows that $\model, w \sat \varphi$ and $\model, \msuccfct(w) \sat \varphi \U \psi$. Therefore $\model, w \sat \psi \vee \left( \varphi \wedge \tnext \varphi\U\psi\right)$. From right to left, if $\model, w \sat \psi$ then $\model, w \sat \varphi \U \psi$ by definition. If $\model, w \sat \varphi \wedge \tnext \varphi \U \psi$ then $\model, w \sat \varphi$ and $\model,\msuccfct(w)\sat \varphi \U \psi$ so, due to the semantics, we conclude that $\model, w \sat \varphi \U \psi$. In any case, $\model, w \sat \varphi \U \psi$.		

For \eqref{ItUTwo}, we work by contrapositive. From right to left, let us assume that $\model, w \not \sat \varphi \R \psi$. Therefore there exists $k \ge 0$ s.t. $\model, \msuccfct[^k](w) \not\sat \psi$ and for all $j$ satisfying $0 \le j < k$, $\model, \msuccfct[^j](w) \not\sat \varphi$. If $k = 0$  then $\model, w\not \sat \psi$ while, if $k > 0$ it follows that $\model, w \not \sat \varphi$ and $\model, \msuccfct(w) \not \sat \varphi \R \psi$. In any case, $\model, w \not \sat \psi \wedge \left( \varphi \vee \tnext \varphi\R\psi\right)$. From left to right, if $\model, w \not \sat \psi$ then $\model, w \not \sat \varphi \R \psi$ by definition. If $\model, w \not \sat \varphi \vee \tnext \varphi \R \psi$ then $\model, w \not \sat \varphi$ and $\model,\msuccfct(w)\not \sat \varphi \U \psi$ so, due to the semantics of $\R$, we conclude that $\model, w \not \sat \varphi \R \psi$. In any case, $\model, w \not \sat \varphi \R \psi$.

%For \ref{ItUThree}, assume by contradiction that the formula is not valid. Then there exists a $\iltl$ model $\model$ and a Kripke world $w$ s.t. $\model, w \sat \varphi \U \psi$ but $\model, w \not \sat \diam\psi$. From the former assumption it follows that $\model, \msuccfct[^k](w) \sat \psi$ for some $k\ge 0$, which contradicts the latter assumption.
		
%For \ref{ItUFour}, assume by contradiction that the formula is not valid. Then there exists a $\iltl$ model $\model$ and a Kripke world $w$ s.t. $\model, w \sat \ubox \psi$ but $\model, w \not \sat \varphi \R \psi$. From the latter assumption it follows that $\model, \msuccfct[^k](w) \not \sat \psi$ for some $k\ge 0$, which contradicts the former assumption.
								
%For \ref{ItUFive}, from left to right, if $\model, w \sat \diam \varphi$ then $\model,\msuccfct[^k](w) \sat \varphi $ for some $k \ge 0$. Since $\model,\msuccfct[^j](w) \sat \top $ for all $j$ satisfying $0 \le j < k$ it follows that $\model, w \sat \top \U \varphi$. The other direction is straightforward.

%Finally, for \ref{ItUSix}, from left to right, assume by contradiction that  $\model, w \not \sat \bot \R\varphi$ which means that $\model,\msuccfct[^k](w) \not \sat \varphi $ for some $k \ge 0$, which contradicts the asssumption $\model, w \sat \ubox \varphi$. The converse direction is also straightforward.	

The remaining items are left to the reader.
\end{proof}

As in the classical case, over the class of persistent models we can `push down' all occurrences of $\tnext$ to the propositional level.
Say that a formula $\varphi$ is in {\em $\tnext$-normal form} if all occurrences of $\tnext$ are of the form $\tnext ^i p$, with $p$ a propositional variable.

\begin{theorem}\label{TheoNextInside}
Given $\varphi\in \lang$, there exists $\widetilde\varphi$ in $\tnext$-normal form such that $\varphi\leftrightarrow\widetilde\varphi$ is valid over the class of persistent models.
\end{theorem}

\proof
The claim can be proven by structural induction using the validities in Propositions \ref{prop:valid:iltl}, \ref{prop:nvalid:iltl} and \ref{PropUValid}.
\endproof

We remark that the only reason that this argument does not apply to arbitrary $\iltl$ models is the fact that $(\tnext\varphi\to\tnext\psi)\to\tnext(\varphi\to\psi)$ is not valid in general (Proposition \ref{prop:nvalid:iltl}).

\section{Bounded bisimulations for $\diam$ and $\ubox$}

In this section we adapt the classical definition of bounded bisimulations for modal logic~\cite{BRV01} to our case. To do so we combine the ordinary definition of bounded bisimulations with the work of~\cite{P97} on bisimulations for propositional intuitionistic logic. Such work introduces extra conditions involving the partial order $\irel$. In our setting, we combine both approaches in order to define bisimulation for a language involving $\diam$, $\ubox$ and $\tnext$ as modal operators plus an intuitionistic $\imp$. Since all languages we consider contain Booleans and $\tnext$, it is convenient to begin with a `basic' notion of bisimulation for this language.

\begin{definition}
Given $n > 0 $ and two $\iltl$ models $\model_1$ and $\model_2$, a sequence of binary relations $\bisim_n \subseteq \cdots \subseteq \bisim_0 \subseteq W_1 \times W_2$ is said to be a 
\emph{bounded $\tnext$-bisimulation} if for all $(w_1,w_2)\in W_1 \times W_2$ and for all $0\le i < n$, the following conditions are satisfied:
\smallskip

\noindent{\sc Atoms.} If $w_1 \bisim_i w_2$ then for all propositional variables $p$, $\model_1,w_1 \sat p$ iff $\model_2, w_2\sat p$.\smallskip

\noindent{\sc Forth $\to$.} If $w_1 \bisim_{i+1} w_2$ then for all $v_1 \in W_1$, if $v_1 \succcurlyeq w_1$, there exists $v_2\in W_2$ such that $v_2 \succcurlyeq w_2$ and $v_1 \bisim_i v_2$.
\smallskip

\noindent{\sc Back $\to$.} If $w_1 \bisim_{i+1} w_2$ then for all $v_2 \in W_2$ if $v_2\succcurlyeq w_2$ then there exists $v_1\in W_1$ such that $v_1 \succcurlyeq w_1$ and $v_1 \bisim_iv_2$.
\smallskip

\noindent{\sc Forth $\tnext$.} if $w_1 \bisim_{i+1} w_2$ then $\msuccfct(w_1) \bisim_i \msuccfct(w_2)$.
\end{definition}

Note that there is not `back' clause for $\tnext$; this is simply because $S$ is a function, so its `forth' and `back' clauses are identical. Bounded $\tnext$-bisimulations are useful because the preserve the truth of relatively small $\lang_{\tnext}$-formulas.

\begin{lemma}\label{lemma:bisimulation:tnext}
Given two $\iltl$ models $\model_1$ and $\model_2$ and a bounded $\tnext$-bisimulation $\bisim_n \subseteq \cdots \subseteq \bisim_0$ between them, for all $i\leq n$ and $(w_1,w_2)\in W_1\times W_2$, if $w_1\bisim_i w_2$ then for all $\varphi \in \lang_{\tnext}$ satisfying $\length{\varphi} \le i$\footnote{Although not optimal, we use the length of the formula in this lemma to simplify its proof. More precise measures like counting the number of modalities and implications could be equally used.}, $\model_1, w_1 \sat \varphi \hbox{ iff } \model_2, w_2 \sat \varphi$.									
\end{lemma}

\begin{proof}We proceed by induction on $i$. Let $0 \leq i \leq n$ be such that for all $j < i$ the lemma holds. Let $w_1\in W_1$ and $w_2 \in W_2$ be such that $w_1 \bisim_n w_2$ and let us consider $\varphi \in \lang_{\diam}$ such that $\length{\varphi}\le i$. The cases where $\varphi$ is an atom or of the forms $\theta\wedge \psi$, $\theta\vee \psi$ are as in the classical case and we omit them. Thus we focus on the following:\smallskip

%\noindent{\sc Case of an atom $p$}. proved by using Condition~\ref{def:bisim:d:c1}.
%		\item Case $\varphi \wedge \psi$: assume without loss of generality that that $\model_1, w_1 \sat \varphi \wedge \psi$, so $\model_1, w_1 \sat \varphi$ and $\model_1, w_1 \sat\psi$. Since $\length{\varphi\wedge\psi} \le n$ then  $\length{\varphi}< n$  and  $\length{\psi} < n$. By induction hypothesis $\model_2, w_2 \sat \varphi$ and $\model_2, w_2 \sat \psi$ so $\model_2, w_2 \sat \varphi \wedge \psi$.		
%		\item Case $\varphi \vee \psi$: assume without loss of generality that that $\model_1, w_1 \sat \varphi \vee \psi$, so either $\model_1, w_1 \sat \varphi$ or $\model_1, w_1 \sat\psi$. Since $\length{\varphi\vee\psi} \le n$ then  $\length{\varphi}< n$  and  $\length{\psi} < n$ so, by induction hypothesis, either $\model_2, w_2 \sat \varphi$ or $\model_2, w_2 \sat \psi$. Therefore $\model_2, w_2 \sat \varphi \vee \psi$.
				
\noindent{\sc Case $\varphi = \theta \imp \psi$.} We proceed by contrapositive to prove the left-to-right implication. Note that in this case we must have $i>0$.

Assume that $\model_2, w_2 \not \sat \theta \imp \psi$. Therefore there exists $v_2 \in W_2$ such that $v_2 \succcurlyeq w_2$, $\model_2, v_2 \sat \theta $, and $\model_2, v_2 \not \sat \psi$. By the {\sc Back $\to$} condition, it follows that there exists $v_1\in W_1$ such that $v_1\succcurlyeq w_1$ and $v_1 \bisim_{i-1} v_2$. Since $\length{\theta} \leq i-1$ and $\length{\psi} < n$, by the induction hypothesis, it follows that $\model_1,v_1 \sat \theta$ and $\model_1,v_1 \not \sat \psi$. Consequently, $\model_1, w_1 \not\sat \theta \imp \psi$. The converse direction is proved in a similar way but using the {\sc Forth $\imp$.}\smallskip

\noindent {\sc Case $\varphi = \tnext \psi$.} Once again we have that $i>0$. Assume that $\model_1, w_1\sat \tnext \psi$, so that $\model_1, \msuccfct(w_1)\sat \psi$. By {\sc Forth $\tnext$}, $S_1(w_1) \bisim_{i-1} S_2(w_2)$. Moreover, $\length{\psi} \leq i-1$, so that by the induction hypothesis, $\model_2, \msuccfct(w_2)\sat \psi$, and $\model_2, w_2\sat \tnext \psi$. The right-to-left direction is analogous.
\end{proof}	

Next, we will extend the notion of a bounded $\tnext$-bisimulation to include other tenses. Let us begin with $\diam$.

\begin{definition}
Given $n > 0$ and two $\iltl$ models $\model_1$ and $\model_2$, a bounded $\tnext$-bisimulation $\bisim_n \subseteq \cdots \subseteq \bisim_0 \subseteq W_1 \times W_2$ is said to be a 
\emph{bounded $\diam$-bisimulation} if for all $(w_1,w_2)\in W_1 \times W_2$ and for all $0\le i < n$, if $w_1 \bisim_{i+1} w_2$, then the following conditions are satisfied:
\smallskip
	
\noindent{\sc Forth $\diam$.} For all $k_1\ge 0$ there exist $k_2 \ge 0$ and $(v_1,v_2) \in W_1\times W_2$ such that $ \msuccfct[^{k_2}](w_2) \succcurlyeq v_2$, $v_1 \succcurlyeq \msuccfct[^{k_1}](w_1)$ and $v_1 \bisim_i v_2$.\smallskip

\noindent{\sc Back $\diam$.} For all $k_2\ge 0$ there exist $k_1 \ge 0$ and $(v_1,v_2) \in W_1\times W_2$ such that $\msuccfct[^{k_1}](w_1) \succcurlyeq v_1$, $v_2 \succcurlyeq \msuccfct[^{k_2}](w_2)$ and $v_1 \bisim_i v_2$.\smallskip
\end{definition}

As was the case of Lemma \ref{lemma:bisimulation:tnext}, if two worlds are related by a bounded $\diam$-bisimulation, then they satisfy the same $\lang_{\diam}$-formulas of small length.

\begin{lemma}\label{lemma:bisimulation:diam}
Given two $\iltl$ models $\model_1$ and $\model_2$ and a bounded $\diam$-bisimulation $\bisim_n \subseteq \cdots \subseteq \bisim_0$ between them, for all $i\leq n$ and $(w_1,w_2)\in W_1\times W_2$, if $w_1\bisim_n w_2$ then for all\footnote{We remind the reader that, as per our convention, $\lang_{\diam}$ is the $\ubox,{\U},{\R}$-free fragment. A similar comment applies to other sublanguages of $\lang$ mentioned below.} $\varphi \in \lang_{\diam}$ satisfying $\length{\varphi} \le n$, $\model_1, w_1 \sat \varphi \hbox{ iff } \model_2, w_2 \sat \varphi$.									
\end{lemma}

\begin{proof}We proceed by induction on $n$. Let $0 \le i \le n$ be such that for all $j < i$ the lemma holds. Let $w_1\in W_1$ and $w_2 \in W_2$ be such that $w_1 \bisim_i w_2$ and let us consider $\varphi \in \lang_{\diam}$ such that $\length{\varphi}\le i$. We only consider the case where $\varphi = \diam \psi$, as other cases are covered by Lemma \ref{lemma:bisimulation:tnext}.

From left to right, if $\model_1,w_1\sat \diam \psi$ then there exists $k_1 \ge 0$ such that $\model_1,\msuccfct[^{k_1}](w_1)\sat \psi$. By {\sc Forth $\diam$}, there exists $k_2 \ge 0$ and $(v_1,v_2)\in W_1\times W_2$ such that $\msuccfct[^{k_2}](w_2) \succcurlyeq v_2 $, $ v_1 \succcurlyeq \msuccfct[^{k_1}](w_1)$ and $v_1 \bisim_{i-1} v_2$. By $\peq$-monotonicity, $\model_1,v_1 \sat \psi$. Then, by the induction hypothesis and the fact that $\length{\psi} \leq i-1$, it follows that $\model_2,v_2\sat \psi$, thus by $\peq$-monotonicity once again, $\model_2, \msuccfct[^{k_2}](w_2)\sat \psi$, so that $\model_2, w_2\sat \diam\psi$. The converse direction is proved similarly by using {\sc Back $\diam$}.
\end{proof}	

We can define bounded $\ubox$-bisimulations in a similar way.

\begin{definition}	
A bounded $\tnext$-bisimulation $\bisim_n \subseteq \cdots \subseteq \bisim_0 \subseteq W_1 \times W_2$ is said to be a \emph{bounded $\ubox$-bisimulation} if for all $(w_1,w_2)\in W_1 \times W_2 $ and for all $0\le i < n$, if $w_1 \bisim_{i+1} w_2$, then:

\noindent
{\sc Forth $\ubox$.} For all $k_2\ge 0$ there exist $k_1 \ge 0$ and $(v_1,v_2) \in W_1\times W_2$ s.t. $\msuccfct[^{k_2}](w_2)\succcurlyeq v_2$, $v_1 \succcurlyeq \msuccfct[^{k_1}](w_1)$ and $v_1 \bisim_i v_2$.\smallskip

\noindent
{\sc Back $\ubox$.} For all $k_1\ge 0$	 there exist $k_2 \ge 0$ and $(v_1,v_2) \in W_1\times W_2$ s.t. $\msuccfct[^{k_1}](w_1)\succcurlyeq v_1$, $v_2 \succcurlyeq \msuccfct[^{k_2}](w_2)$ and $v_1 \bisim_i v_2$.
\end{definition}

\begin{lemma}\label{lemma:bisimulation:ubox}					
Given two $\iltl$ models $\model_1$ and $\model_2$ and a bounded $\ubox$-bisimulation $\bisim_n \subseteq \cdots \subseteq \bisim_0$ between them, for all $(w_1,w_2)\in W_1\times W_2$ and $i\leq n$, if $w_1\bisim_i w_2$ then for all $\varphi \in \lang_{\ubox}$ such that $\length \varphi \le i$, then $\model_1, w_1 \sat \varphi \hbox{ iff } \model_2, w_2 \sat \varphi$.									
\end{lemma}

\begin{proof} We proceed by induction on $i$. Let $i \ge 0$ be such that for all $j < i$ the lemma holds. Let $w_1\in W_1$ and $w_2 \in W_2$ be such that $w_1 \bisim_i w_2$ and let us consider $\varphi \in \lang_{\ubox}$ such that $\length{\varphi} \le i$. Note that the cases for atoms as well as propositional and $\tnext$ connectives are proved as in Lemma~\ref{lemma:bisimulation:tnext}, so we only consider $\varphi = \ubox\psi$. 

For the left-to-right implication, we work by contrapositive, and assume that $\model_2,w_2 \not\sat \ubox \psi$. Then, there exists $k_2 \ge 0$ such that $\model_2,\msuccfct[^{k_2}](w_2)\not \sat \psi$. By {\sc Forth $\ubox$}, there exist $k_1 \ge 0$ and $(v_1,v_2) \in W_1\times W_2$ s.t. $\msuccfct[^{k_2}](w_2)\succcurlyeq v_2$, $v_1 \succcurlyeq \msuccfct[^{i_1}](w_1)$ and $v_1 \bisim_{i-1} v_2$. As in the proof of Lemma \ref{lemma:bisimulation:diam}, by $\peq$-monotonicity, the induction hypothesis and the fact that $\length{\psi} \leq i-1$, it follows that $\model_1,v_1\not \sat \psi$; thus $\model_1, \msuccfct[^{k_1}](w_1)\not \sat \psi$, and again by $\peq$-monotonicity $\model_1, w_1\not \sat \ubox\psi$. The converse direction follows a similar reasoning but using {\sc Back $\ubox$}. 
\end{proof}

\section{Bounded bisimulations for $\U$ and $\R$}

In this section we adapt the bisimulations defined for a language with \emph{until} and \emph{since}~\cite{Kamp68} presented by Kurtonina and de Rijke~\cite{KR97} to our case. Let us begin with bounded bisimulations for $\U$.

\begin{definition}
Given $n\in \mathbb{N}$ and two $\iltl$ models $\model_1$ and $\model_2$, a bounded $\tnext$-bisimulation $\bisim_n \subseteq \cdots \subseteq \bisim_0 \subseteq W_1 \times W_2$ is said to be a \emph{bounded $\U$-bisimulation} iff for all $(w_1,w_2)\in W_1 \times W_2, $ $w_1 \bisim_n w_2$ and for all $0\le i < n$:\smallskip

\noindent{\sc Forth $\U$.} For all $k_1\ge 0$ there exist $k_2 \ge 0$ and $(v_1,v_2) \in W_1\times W_2$ such that
\begin{enumerate}[itemsep=0pt]

\item $\msuccfct[^{k_2}](w_2) \succcurlyeq v_2$, $v_1 \succcurlyeq \msuccfct[^{k_1}](w_1)$ and $v_1 \bisim_i v_2$, and

\item  for all $j_2 \in [0,k_2)$  there exist $j_1 \in [0,k_1)$ and $(u_1,u_2) \in W_1 \times W_2$ such that $u_1 \succcurlyeq \msuccfct[^{j_1}](w_1)$, $\msuccfct[^{j_2}](w_2) \succcurlyeq u_2$ and $u_1\bisim_i u_2$.\smallskip

\end{enumerate}

\noindent{\sc Back $\U$.} For all $k_2\ge 0$ there exist $k_1 \ge 0$ and $(v_1,v_2) \in W_1\times W_2$ such that
\begin{enumerate}[itemsep=0pt]

\item

$\msuccfct[^{k_1}](w_1) \succcurlyeq v_1$, $v_2 \succcurlyeq \msuccfct[^{k_2}](w_2)$ and $v_1 \bisim_i v_2$, and

\item for all $j_1 \in [0,k_1)$  there exist $j_2 \in [0,k_2)$ and $(u_1,u_2) \in W_1 \times W_2$ such that $u_2 \succcurlyeq \msuccfct[^{j_2}](w_2)$, $\msuccfct[^{j_1}](w_1) \succcurlyeq u_1$ and $u_1\bisim_i u_2$.

\end{enumerate}
\end{definition}

As was the case before, the following lemma states that two bounded $\U$-bisimilar models agree on small $\lang_{\U}$ formulas.

\begin{lemma}\label{lemma:bisimulation:until}					
	Given two $\iltl$ models $\model_1$ and $\model_2$ and a bounded $\U$-bisimulation $\bisim_n\subset \hdots \subset \bisim_0$ between them, for all $m\leq n$ and $(w_1,w_2)\in W_1\times W_2$, if $w_1\bisim_m w_2$ then for all $\varphi \in \lang_{\U}$ such that $\length{\varphi} \le n$, $\model_1, w_1 \sat \varphi \hbox{ iff } \model_2, w_2 \sat \varphi$.									
\end{lemma}

	\begin{proof} Once again, proceed by induction on $n$. Let $m\leq n$ be such that for all $k < m$ the lemma holds. Let $w_1\in W_1$ and $w_2 \in W_2$ be such that $w_1 \bisim_m w_2$ and let us consider $\varphi \in \lang_{\U}$ such that $\length{\varphi}\le m$.
	As before, we only consider the `new' case, where $\varphi = \theta \U\psi$. From left to right, assume that $\model_1, w_1 \sat \varphi \U \psi$.
	Then, there exists $i_1\ge 0$ such that $\model_1, \msuccfct[^{i_1}](w_1) \sat \psi$ and for all $j_1$ satisfying $0 \le j_1 < i_1$, $\model_1, \msuccfct[^{j_1}](w_1) \sat \varphi$.
	By {\sc Forth $\U$}, there exist $i_2 \ge 0$ and $(v_1,v_2) \in W_1\times W_2$ such that
	\begin{enumerate*}[itemsep=0pt]
		\item $\msuccfct[^{i_2}](w_2)\succcurlyeq v_2$, $v_1 \succcurlyeq \msuccfct[^{i_1}](w_1)$ and $v_1 \bisim_{m-1} v_2$;
		\item for all $j_2$ satisfying $0 \le j_2 < i_2$  there exist $j_1 \in [0,i_1)$ and $(u_1,u_2) \in W_1 \times W_2$ s. t. $u_1 \succcurlyeq \msuccfct[^{j_1}](w_1)$, $\msuccfct[^{j_2}](w_2) \succcurlyeq u_2$ and $u_1\bisim_{m-1} u_2$.
	\end{enumerate*}
		
From the first item, $\peq$-monotonicity, the fact that $\length{\psi} \leq m-1$, and the induction hypothesis, it follows that $\model_2, \msuccfct[^{i_2}](w_2) \sat \psi$. Take any $j_2$ satisfying $0 \le j_2 < i_2$. By the second item, the fact that $\length{\theta} \leq m-1$, and the induction hypothesis, we conclude that $\model_2, \msuccfct[^{j_2}](w_2) \sat \varphi$ so $\model_2, w_2 \sat \varphi \U\psi$. The right-to-left direction is symmetric (but using {\sc Back $\U$}).				
	\end{proof}	

Finally, we define bounded bisimulations for $\R$.

\begin{definition}
A bounded $\tnext$-bisimulation $\bisim_n \subseteq \cdots \subseteq \bisim_0 \subseteq W_1 \times W_2$ is said to be a \emph{bounded $\R$-bisimulation} if for all $(w_1,w_2)\in W_1 \times W_2 $ such that $w_1 \bisim_{i+1} w_2$ and for all $0\le i < n$:\smallskip

\noindent{\sc Forth $\R$.} For all $k_2\ge 0$ there exist $k_1 \ge 0$ and $(v_1,v_2) \in W_1\times W_2$ such that

\begin{enumerate}[itemsep=0pt]

\item $\msuccfct[^{k_2}](w_2)\succcurlyeq v_2$, $v_1 \succcurlyeq \msuccfct[^{k_1}](w_1)$ and $v_1 \bisim_i v_2$, and

\item for all $j_1$ satisfying $0\le j_1 < k_1$ there exist $j_2$ such that $0 \le j_2 < k_2$ and $(u_1,u_2) \in W_1 \times W_2$ s. t. $u_1 \succcurlyeq \msuccfct[^{j_1}](w_1)$, $\msuccfct[^{j_2}](w_2)\succcurlyeq u_2$ and $u_1\bisim_i u_2$.

\end{enumerate}

\noindent{\sc Back $\R$.} For all $k_1\ge 0$ there exist $k_2 \ge 0$ and $(v_1,v_2) \in W_1\times W_2$ such that

\begin{enumerate}[itemsep=0pt]

\item

$\msuccfct[^{k_1}](w_1)\succcurlyeq v_1$, $v_2 \succcurlyeq \msuccfct[^{k_2}](w_2)$ and $v_1 \bisim_i v_2$, and

\item for all $j_2$ satisfying $0 \le j_2 < k_2$ there exist $j_1$ such that $0 \le j_1 < k_1$ and $(u_1,u_2) \in W_1 \times W_2$ s. t. $u_2 \succcurlyeq \msuccfct[^{j_2}](w_2)$, $\msuccfct[^{j_1}](w_1)\succcurlyeq u_1$ and $u_1\bisim_i u_2$. 
\end{enumerate}
\end{definition}

Once again, we obtain a corresponding bisimulation lemma for $\lang_{\R}$.

\begin{lemma}\label{lemma:bisimulation:release}					
Given two $\iltl$ models $\model_1$ and $\model_2$ and a bounded $\R$-bisimulation $\bisim_n \subseteq \cdots \subseteq \bisim_0$ between them, for all $m\leq n$ and $(w_1,w_2)\in W_1\times W_2$, if $w_1\bisim_m w_2$ then for all $\varphi \in \lang_{\U}$ such that $\length{\varphi} \le m$, $\model_1, w_1 \sat \varphi \hbox{ iff } \model_2, w_2 \sat \varphi$.									
\end{lemma}

\begin{proof}
As before, we proceed by induction on $n$; the critical case where $\varphi =\theta \R \psi$ follows by a combination of the reasoning for Lemmas \ref{lemma:bisimulation:ubox} and Lemma \ref{lemma:bisimulation:ubox}. Details are left to the reader.
	\end{proof}

\section{Definability and undefinability of modal operators} \label{SecSucc}

In this section, we explore the question of when it is that the basic connectives can or cannot be defined in terms of each other.
It is known that, classically, $\diam$ and $\ubox$ are interdefinable, as are $\U$ and $\R$; we will see that this is not the case intuitionistically. On the other hand, $\U$ (and hence $\R$) is not definable in terms of $\diam,\ubox$ in the classical setting \cite{Kamp68}, and this result immediately carries over to the intuitionistic setting, as the class of classical $\sf LTL$ models can be seen as the subclass of that of dynamic posets where the partial order is the identity.

Interdefinability of modal operators can vary within intermediate logics. For example, $\wedge$, $\vee$ and $\rightarrow$ are basic connectives in propositional intuitionistic logic, but in the intermediate logic of here-and-there~\cite{Hey30}, $\wedge$~\cite{A+15,BalbianiDieguezJelia} and $\rightarrow$~\cite{A+15} are basic operators while $\vee$ is definable in terms of $\rightarrow$ and $\wedge$~\cite{LK41}. In first-order here-and-there~\cite{L+07}, the quantifier $\exists$ is definable in terms of $\forall$ and $\rightarrow$~\cite{Mints10} while $\forall$ is not definable in terms of the other operators. In the modal case, Simpson~\cite{Simpson94} shows that modal operators are not interdefinable in the logic $\sf IK$ and Balbiani and Di\'eguez~\cite{BalbianiDieguezJelia} proved the same result for the linear time temporal extension of here-and-there. This last proof is adapted to show that modal operators are not definable in $\iltl$. Note, however, that here we correct the claim of \cite{BalbianiDieguezJelia} stating that $\diam$ is not here-and-there definable in terms of $\ubox$.

{\color{black}

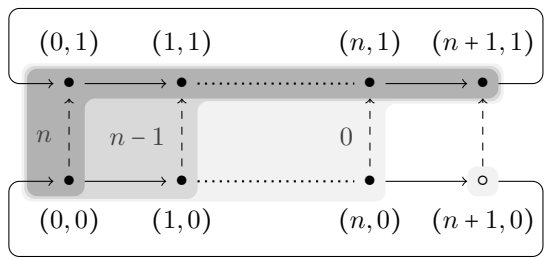
\begin{figure}
\begin{center}
\begin{tikzpicture}[label distance=1pt]
  \foreach \x/\n/\c in {0/0/n, 1.5/1/n-1, 4/n/0}
  { \path (\x,0) node ["${(\n,0)}$" below] (n\n0) {$\bullet$};
    \path (\x,1.3) node ["${(\n,1)}$" above] (n\n1) {$\bullet$};
    \path (\x,0.6) -- +(-.1,0) node [anchor=east,color=black!50!gray] (c\n) {$\c$};
  }
  \path (5.5,0) node ["${(n+1,0)}$" below] (np0) {$\circ$};
  \path (5.5,1.3) node ["${(n+1,1)}$" above] (np1) {$\bullet$};
  \begin{scope}[on background layer]
    \foreach \n/\g/\D in {n/10/2pt, 1/30/1pt, 0/60/0}
    \fill[fill=gray!\g,rounded corners]
         ($(n\n0.south east) - (0,\D)$)
      -| ($(c0.west) - (\D,0)$) |- ($(np1.north east) + (\D,\D)$)
      |- ($(n\n1.south east) - (0,\D)$)
      -- cycle;
    \fill[fill=gray!10,rounded corners] (np0.north west) rectangle (np0.south east);
  \end{scope}

  \foreach \n in {0, 1, n, p}
  \draw[dashed,->] (n\n0) -> (n\n1);

  \foreach \y in {0, 1}
  { \foreach \f/\t in {0/1, n/p}
      \draw[->] (n\f\y) -> (n\t\y);
    \draw[dotted,thick] (n1\y) -- (nn\y); }
    
  \draw[->,rounded corners] (np0) -| ++(.8,-1) -| (-.8,0) -- (n00);
  \draw[->,rounded corners] (np1) -| ++(.8,1) -| (-.8,1.3) -- (n01);
\end{tikzpicture}

\end{center}
\caption{The here-and-there model $\mathcal H_n$. Black dots satisfy the atom $p$, white dots do not; all other atoms are false everywhere. Dashed lines indicate $\peq$ and solid lines indicate $S$. The $\peq_i$-equivalence classes are shown as grey regions.}\label{FigBoxU}
\end{figure}

Let us begin by studying the definability of $\ubox$ in terms of $\tnext$ and $\U$. Below, if $\lang'\subseteq \lang$, $\varphi\in \lang$ and $\Omega$ is a class of models, we say that $\varphi$ is {\em $\lang'$-definable over $\Omega$} if there is $\varphi' \in \lang'$ such that $\Omega \models \varphi\leftrightarrow \varphi'$.

\begin{theorem}
The connective $\ubox$ is not $\lang_{\U}$-definable, even over the class of finite here-and-there models.
\end{theorem}

\begin{proof}
For $n>0$ consider a model $\mathcal H_n = (W,{\peq},S,V)$ with $W = \{0,\hdots ,n +1 \}\times \{0,1\}$, $(i,j) \peq (i',j')$ if $i = i'$ and $j \leq j'$, $S(i,j) = (i',j')$ if and only if $i'=i$ and $j'\equiv j +1 \pmod{n+2}$, and $V(p) = W \setminus \{(n+1,0)\}$. Clearly $\mathcal H_n$ is a here-and-there model. For $m \leq n$, let $\sim_m$ be the least equivalence relation such that $(i,j) \sim_{m} (i',j')$ whenever
\[\min \{i(j-1) ,i' (j' - 1) \} \leq n - m\]
(see Figure \ref{FigBoxU}). Then, it can easily be checked that $(\mathcal H_n,(0,0)) \not \models \diam p$, $(\model,(0,1)) \models \diam p$, and
$(0,0)\sim_m(0,1)$.

It remains to check that $(\sim_m)_{m\leq n}$ is a bounded $\U$-bisimulation. The atoms, $\imp$ and $\tnext$ clauses are easily verified, so we focus on those for $\U$. Since $\sim_m$ is symmetric, we only check the {\sc Forth $\U$}. Suppose that $(i_1,j_1) \sim_m (i_2,j_2)$, and fix $k_1 \geq 0$. Let  $S(i_1,j_1) = (i',j')$.
Then, we can see that $k_2 = 0$, $v_1 = (i',1)$ and $v_2 = (i_2,j_2)$ witness that the clause holds, where the intermediate condition for $j_2 \in [0,k_2)$ holds vacuously since $[0,k_2) = \varnothing$.

By letting $n = |\varphi|$, we see using Lemma \ref{lemma:bisimulation:until} that that no $\lang_{ \U}$-formula $\varphi$ can be equivalent to $\ubox p$.
\end{proof}

\noindent As a consequence:
\begin{corollary}
The connective $\R$ is not definable in terms of $\tnext$ and $\U$, even over the class of persistent models.
\end{corollary}	

\begin{proof}
If we could define $q\R p$, then we could also define $\ubox p \equiv \bot \R p$.\end{proof}
}

\begin{proposition}\label{propDiamDefinHT}
Over the class of here-and-there models, $\diam$ is $\lang_{\ubox}$-definable. To be precise, $\diam p$ is equivalent to
\[\varphi = (\ubox (p\to \ubox (p\vee \neg p))\wedge \ubox (\tnext \ubox (p\vee \neg p) \to p\vee \neg p\vee \tnext \ubox \neg p))\to (\ubox(p\vee \neg p)\wedge \neg \ubox \neg p).\]
\end{proposition}

\proof
Let $\model = (T\times \{0,1\},{\peq},S,V)$ be a here-and-there model with $S(t,i) = (f(t),i)$ (see Section \ref{SecHT}). First assume that $x=(x_1,x_2)$ is such that $(\model,x)\models \diam p$. To check that $(\model,x)\models\varphi$, let $x'\seq x$, so that $x'=(x_1,x'_2)$ with $x'_2\geq x_1$, and consider the following cases.\smallskip

\noindent {\sc Case $(\model,x')\models \ubox(p\vee \neg p)$.} In this case, it is easy to see that we also have $(\model,x')\models \neg \ubox \neg p$ given that $(\model , x) \models \diam p$.\smallskip

\noindent{\sc Case $(\model,x') \not \models \ubox(p\vee \neg p)$.} Using the assumption that $(\model,x)\models \diam p$, choose $k$ such that $(\model, (f^k(x_1),x_2))\models p$ and consider two sub-cases.\\

\begin{enumerate}[itemsep=0pt]

\item\label{CaseA} Suppose there is $k'>k$ such that $(\model, (f^{k'}(x_1),x'_2))\not\models p\vee \neg p$. Then, it follows that $(\model, (f^{k}(x_1),x'_2))\not\models p \to \ubox p\vee \neg p$ and hence $ (\model,x')\not \models  \ubox (p\to \ubox (p\vee \neg p))$.

\item\label{CaseB} If there is not such $k'$, then there must be a maximal $k'<k$ such that $(\model, (f^{k'}(x_1) ,x'_2))\not\models p\vee \neg p$ (otherwise, we would be in {\sc Case $(\model,x')\models \ubox(p\vee \neg p)$.}). It is easily verified that
\[(\model, ( f^{k'} (x_1) ,x'_2))\not \models  \tnext \ubox (p\vee \neg p) \to p\vee \neg p\vee \tnext \ubox \neg p ,\]
and hence
\[(\model, x' )\not \models \ubox (\tnext \ubox (p\vee \neg p) \to p\vee \neg p\vee \tnext \ubox \neg p).\]

\end{enumerate}

Note that the above direction does not use any properties of here-and-there models, and works over arbitrary expanding models. However, we need these properties for the other implication. Suppose that $(\model,x) \models \varphi$. If $ (\model,x) \models \ubox(p\vee \neg p)\wedge \neg \ubox \neg p$, then it is readily verified that $ (\model,x) \models \diam p$. Otherwise,
\[ (\model,x) \not \models \ubox (p\to \ubox (p\vee \neg p))\wedge \ubox (\tnext \ubox (p\vee \neg p) \to p\vee \neg p\vee \tnext \ubox \neg p).\]
If $(\model,x) \not \models \ubox (p\to \ubox (p\vee \neg p))$, then there is $k$ such that
\[(\model,(f ^k(x_1) ,x_2))\not \models p\to \ubox (p\vee \neg p).\]
This is only possible if $x_2 = 0$ and $(\model,(f^k(x_1),x_2))\models p$, so that $(\model, x)\models \diam p$. Similarly, if
\[(\model,x) \not \models \ubox (\tnext \ubox (p\vee \neg p) \to p\vee \neg p\vee \tnext \ubox \neg p),\]
then there is $k$ such that $(\model,(f^k(x_1) ,x_2))\not \models \tnext \ubox (p\vee \neg p) \to p\vee \neg p\vee \tnext \ubox \neg p$. This is only possible if $x_2 = 0$, $(\model,( f^ k (x_1) ,x_2)) \models \tnext \ubox (p\vee \neg p)$ and $(\model,( f^k(x_1) ,x_2))\not \models \tnext \ubox \neg p$. But from this it easily can be seen that there is $k'>k$ with $(\model,( f^{k'}(x_1),x_2)) \models p$, hence $(\model,x)\models \diam p$.
\endproof

\begin{corollary}
Over the class of here-and-there models, $p\U q$ is $\lang_{\R}$-definable using the equivalence $p\U q \equiv (q \R (p\vee q))\wedge \diam p$. 
\end{corollary}

Hence, if we want to prove the undefinability of $\diam$ in terms of other operators, we must turn to a wider class of models, as we will do next.

\begin{theorem}
The operator $\diam$ cannot be defined in terms of $\ubox$ over the class of finite expanding models.
\end{theorem}

\begin{figure}
\begin{center}

\begin{tikzpicture}[label distance=1pt]
  \foreach \x/\n/\c in {0/0/n, 1.5/1/n-1, 4/n/0}
  { \path (\x,0) node ["${(\n,0)}$" below] (n\n0) {$\circ$};
    \path (\x,1.3) node ["${(\n,1)}$" above] (n\n1) {$\circ$};
    \path (\x,0.65) -- +(-.1,0) node [anchor=east,color=black!50!gray] (c\n) {$\c$};
  }
  \path (5.5,0) node ["${(n+1,0)}$" below] (np0) {$\circ$};
  \path (5.5,1.3) node ["${(n+1,1)}$" above] (np1) {$\bullet$};
  \begin{scope}[on background layer]
    \fill[fill=gray!10,rounded corners] (np1.north west) rectangle (np1.south east);
    \fill[fill=gray!10,rounded corners] (np0.north west) rectangle (np0.south east);
    \foreach \n/\g/\D in {n/10/2pt, 1/30/1pt, 0/60/0}
      \fill[fill=gray!\g,rounded corners]
        ($(n\n0.south east) - (0,\D)$) -| ($(c0.west) - (\D,0)$) |- ($(n\n1.north east) + (0,\D)$)
        -- cycle;
  \end{scope}

  \foreach \n in {0, 1, n, p}
  \draw[dashed,->] (n\n0) -> (n\n1);

  \foreach \y in {0, 1}
  { \foreach \f/\t in {0/1, n/p}
      \draw[->] (n\f\y) -> (n\t\y);
    \draw[dotted,thick] (n1\y) -- (nn\y); }
    
  \draw[->,rounded corners] (np0) -| ++(.8,-1) -| (-.8,0) -- (n00);
  \draw[->,rounded corners] (np1) -| ++(1.2,-2.7) -| (-1.5,0) -- (-.9,0);
\end{tikzpicture}

\end{center}
\caption{The expanding model $\mathcal E_n$. Notation is as in Figure \ref{FigBoxU}.}
\label{FigDiamBox}
\end{figure}
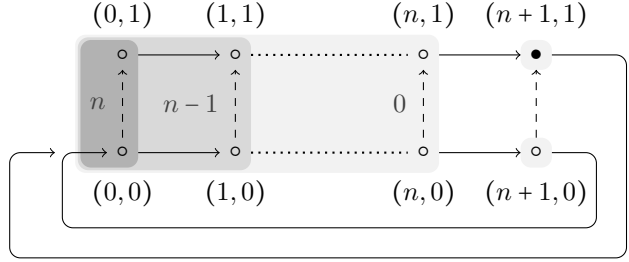

\proof
Given $n>0$, consider a model $\mathcal E_n =(W,{\peq},S,V)$ with $W = \{0,\hdots ,n +1 \}\times \{0,1\}$, $(i,j) \peq (i',j')$ if $i = i'$ and $j \leq j'$, $S(i,j) = (i+1,j)$ if $i \leq n$, $S(n + 1,j) = (0,0)$, and $V(p) = \{(n+1,1)\}$. For $m \leq n$, let $\sim_m$ be the least equivalence relation such that $(i,j) \sim_{m} (i,j')$ whenever $\max \{ i,i'\} \leq n-m$. Then, it can easily be checked that $(\model,(0,0)) \not \models \diam p$, $(\model,(0,1)) \models \diam p$, and $(0,0)\sim_m(0,1)$.

It remains to check that $(\sim_m)_{m \leq n}$ is a bounded $\ubox$-bismulation.
As before, we focus on the $\ubox$ clauses, specifically {\sc Back $\ubox$}. Suppose that $(i_1,j_1)\sim _m (i_2,j_2)$ and fix $k_1\geq 0$. Let $(i'_1, j'_1) = S^{k_1}(i_1,j_1)$. Choose $k_2 > n+1$ such that $i_2 + k_2 \equiv i'_1 \pmod{n+1}$, and let $(i'_2,j'_2) = S^{k_2}(i_2,j_2)$. It is not hard to check that $i'_1 = i'_2$ and $j'_2=0$, from which we obtain $(i'_2,j'_2) \peq (i'_1,j'_1)$. Hence, setting $v_1 = v_2 = (i'_2,j'_2)$ gives us the desired witnesses.

By letting $n$ vary, we see that no $\lang_{\ubox}$-formula can be equivalent to $\diam p$.
\endproof

\section{Conclusions}

In this paper we have investigated on $\iltl$, an intuitionistic analogue of $\sf LTL$ based on expanding domain models from modal logic. We have shown that, as happens in other modal intuitionistic logics or modal intermediate logics, modal operators are not interdefinable. 

This work and \cite{Boudou2017} are the first attempts to study $\iltl$. Needless to say, many open questions remain. We know that $\iltl$ is decidable, but the proposed decision procedure is non-elementary. However, there seems to be little reason to assume that this is optimal, raising the following question:

\begin{question}
Are the satisfiability and validity problems for $\iltl$ without $\U,\R$ elementary? Is the full logic still decidable?
\end{question}

Meanwhile, we saw in Theorems \ref{TheoBound} and \ref{TheoPersNonFin} that $\iltl$ has the strong finite model property, while $\itlb$ does not have the finite model property at all. However, it may yet be that $\itlb$ is decidable despite this.

\begin{question}
Is $\itlb$ decidable?
\end{question}

Regarding expressive completeness, it is known that $\sf LTL$ is expressively complete~\cite{Kamp68,Rabinovich12,Gabbay80,Hodkinson94}; there exists a one-to-one correspondence (over $\mathbb N$) between the temporal language and the monadic first-order logic equipped with a linear order and `next' relation~\cite{Gabbay80}. It is not known whether the same property holds between $\iltl$ and first-order intuitionistic logic.

\begin{question}
Is $\lang$ equally expressive to monadic first-order logic over the class of dynamic or persistent models?
\end{question}

Finally, a sound and complete axiomatization for $\iltl$ remains to be found. The results we have presented here could be a first step in this direction, and we conclude with the following:

\begin{question}
Are the $\iltl$-valid formulas listed in this work, together with the intuitionistic tautologies and standard inference rules, complete for the class of dynamic posets? Is the logic augmented with $(\tnext p\to \tnext q)\to \tnext(p\to q)$ complete for the class of persistent models?
\end{question}

%\bibliographystyle{plain}

%\bibliography{biblio}
\end{document}